\documentclass[11pt,reqno]{amsart}
\usepackage{hyperref}
\usepackage{graphicx}
\usepackage{color}
\usepackage[all]{xy}
\usepackage{amsfonts,amssymb}
\usepackage{amsthm}
\usepackage{bbm} 

\usepackage{mathrsfs}

\usepackage{pdfsync}

\usepackage{a4wide}

\newtheorem{neu}{}[section]
\newtheorem{Cor}[neu]{Corollary}
\newtheorem*{Cor*}{Corollary}
\newtheorem{Thm}[neu]{Theorem}
\newtheorem*{Thm*}{Theorem}

\newtheorem{Prop}[neu]{Proposition}
\newtheorem*{Prop*}{Proposition}
\theoremstyle{definition}
\newtheorem{Lemma}[neu]{Lemma}

\newtheorem*{Rmk*}{Remark}
\newtheorem{Rmk}[neu]{Remark}
\newtheorem{Ex}[neu]{Example}
\newtheorem*{Ex*}{Example}

\newtheorem*{Qu*}{Question}

\newtheorem{Def}[neu]{Definition}

\newcommand{\N}{\mathbb{N}}

\newcommand{\Z}{\mathbb{Z}}
\newcommand{\R}{\mathbb{R}}
\newcommand{\HH}{\mathbb{H}}
\newcommand{\C}{\mathbb{C}}

\newcommand{\om}{\omega}
\newcommand{\Om}{\Omega}

\renewcommand{\S}{\mathbb{S}}
\renewcommand{\P}{\mathcal{P}}

\renewcommand{\L}{\mathcal{L}}

\newcommand{\beq}{\begin{equation}}
\newcommand{\beqn}{\begin{equation}\nonumber}
\newcommand{\eeq}{\end{equation}}

\newcommand{\bea}{\begin{equation}\begin{aligned}}
\newcommand{\bean}{\begin{equation}\begin{aligned}\nonumber}
\newcommand{\eea}{\end{aligned}\end{equation}}

\numberwithin{equation}{section}

\definecolor{Urs}{rgb}{0,.7,0}
\definecolor{Youngjin}{rgb}{0,0,1}
\definecolor{red}{rgb}{1,0,0}

\newcommand{\p}{\partial}

\newcommand{\T}{{\mathbb{T}}}		
\newcommand{\one}{{1\hskip-2.5pt{\rm l}}}

\begin{document}
\title[The growth sequence of symplectomorphisms]
{The growth sequence of symplectomorphisms\\
on symplectically hyperbolic manifolds}
\author{Youngjin Bae}
\address{
    Youngjin Bae\\
    Center for Geometry and Physics, Institute for Basic Science
	and Pohang University of Science and Technology (POSTECH)\\
	77 Cheongam-ro, Nam-gu, Pohang-si, Gyeongsangbuk-do, Korea 790-784}

\email{yjbae@ibs.re.kr}
\begin{abstract}
We study the growth rate of a sequence 
which measures the uniform norm of 
the differential under the iterates of maps.
On symplectically hyperbolic manifolds,
we show that this sequence has at least linear growth
for every non-identical symplectomorphisms
which are symplectically isotopic to the identity. 
\end{abstract}
\maketitle

\section{Introduction}
Let $(M,g)$ be a closed connected Riemannian manifold.
Given a diffeomorphism $\varphi$ on $M$, 
its {\em growth sequence} is defined by
\bean
\Gamma_n(\varphi):=
\max\left(\max_{m\in M}|d\varphi^n(m)|_g,
\max_{m\in M}|d\varphi^{-n}(m)|_g\right),\ n\in\N.
\eea
Here the norm of the differential is given by
the operator norm.
The explicit value of the growth sequence is hard to compute.
With the following terminology, however, 
we measure the growth rate of $\Gamma_n(\varphi)$.
Given two positive sequences 
$a_n,b_n:\N\to[0,\infty)$, 
write $a_n\lesssim b_n$
if there exists $c>0$ such that $a_n\leq c(b_n+1)$
for all $n\in\N$.

We investigate $\Gamma_n(\varphi)$
for symplectomorphisms $\varphi$
in the identity component of the group of
symplectomorphisms $\mathit{Symp}_0(M,\om)$
of a symplectic manifold $(M,\om)$.
Polterovich proved in his beautiful paper \cite{Pol} 
many interesting results 
about the growth type of $\Gamma_n(\varphi)$
for $\varphi\in\mathit{Symp}_0(M,\om)\setminus\{\one\}$
on a closed symplectic manifold $(M,\om)$
with $\pi_2(M)=0$.
Among them, the following result is related to this article:
A closed symplectic manifold $(M,\om)$
with $\pi_2(M)=0$ 
is called {\em symplectically hyperbolic},
if the symplectic form $\om$ admits a bounded primitive 
on the universal cover
$p:\widetilde M\to M$.
For symplectically hyperbolic manifolds $(M,\om)$,
Polterovich showed that $\Gamma_n(\varphi)\gtrsim n$ when 
$\varphi\in{\mathit{Symp}}_0(M,\om)\setminus\{\one\}$
has a fixed point of contractible type.

The existence of fixed points 
with positive action difference is crucial in his proof.
When $\varphi$ is a Hamiltonian diffeomorphism,
the existence of a fixed point is obtained from 
Floer's proof \cite{Fl} of Arnold's conjecture 
and another fixed point with different action
is established in the work of Schwarz \cite{Sch}
by using Floer homology.
For a non-Hamiltonian symplectomorphism,
the flux does not vanish, see Definition \ref{def:flux}.
This guarantees a second fixed point
with different action.

In the 2-dimensional case, i.e. a closed oriented surface $\Sigma_g$ 
of genus $g\geq 2$, see Example \ref{ex:g2},
we know that there exists a fixed point of 
$\varphi\in\mathit{Symp}_0(\Sigma_g,\om)$. 
In \cite{Gro1}, Gromov proved that 
$(-1)^n\chi(M)>0$ for a symplectically hyperbolic manifold 
$(M,\om)$ which is K\"ahler. 
Using this non-vanishing of the Euler characteristic,
one can also obtain a fixed point of contractible type of 
$\varphi\in\mathit{Symp}_0(M,\om)$ 
as in \cite[Example 1.3.C]{Pol}.
In the above cases, we hence conclude 
$\Gamma_n(\varphi)\gtrsim n$
for $\varphi\in\mathit{Symp}_0(M,\om)\setminus\{\one\}$,
but to the best of my knowledge there is no known result 
about the existence of fixed points
of $\varphi$ 
on symplectically hyperbolic manifolds in general.
We will prove the following statement that does not depend on 
the existence of a fixed point.\\[-2mm]

\noindent {\bf Main Theorem.}
{\it Let $(M,\om)$ be a symplectically hyperbolic manifold.
If $\varphi\in\mathit{Symp}_0(M,\om)\setminus\{\one\}$,
then $\Gamma_n(\varphi)\gtrsim n$.}\\[-2mm]

A direct consequence of Main Theorem, followed by \cite[Corollary 1.1.D]{Pol},
gives us obstructions to representations of 
certain discrete groups such as $SL(n,\Z)$ with $n\geq3$
into $\mathit{Symp}_0(M,\om)$
of a symplectically hyperbolic manifold $(M,\om)$.
A more precise statement is the following:
Let $G$ be an irreducible non-uniform lattice in a semi-simple real Lie group of real rank at least two. 
Assume that the Lie group is connected, without compact factors and with finite center. 
Then every homomorphism $G\to\mathit{Symp}_0(M,\om)$ has finite image 
when $(M,\om)$ is a symplectically hyperbolic manifold.
See \cite[Section 1.6]{Pol} for details.

Our main tools are the mapping torus construction
and the cofilling function of the twisted Hamiltonian
structure which will be introduced in the following section.
We interpret the non-trivial flux 
of non-Hamiltonian diffeomorphism
as a certain type of cofilling function.\\



\emph{Acknowledgement} :   
I sincerely thank Urs Frauenfelder for valuable discussions and the word of encouragement.
I appreciate Peter Albers, Felix Schlenk and Jungsoo Kang for their helpful comments.
I am also thankful to anonymous referee for detailed corrections.
This paper is written during a stay
at the University of M\"{u}nster 
and the Institute for Basic Science in POSTECH.
I thank them both for their stimulating working atmosphere
and generous support.

This material is supported by the 
SFB 878 –- Groups, Geometry and Actions,
and by the Research Center Program of IBS in Korea(CA1205).

\section{Preliminaries}

\subsection{Cofilling function}

\begin{Def}[\cite{Gro2}, \cite{Pol}]
Let $\sigma\in\Om^2(M)$ be a closed 2-form such that 
$\widetilde\sigma\in\Om^2(\widetilde M)$ is exact.
Let $\P_\sigma$ be the space of all 1-forms on $\widetilde M$
whose differential is $\widetilde\sigma$.
Pick a point $x\in\widetilde M$ and
denote by $B_{x}(s)$ the ball of radius $s>0$ with respect to $\widetilde g$
which is centered at $x$.
Then the {\em cofilling function} 
$u_\sigma:[0,\infty)\to[0,\infty)$ is defined by
\bean
u_\sigma(s)=u_{\sigma,g,x}(s):=
\inf_{\theta\in\P_\sigma}\sup_{z\in B_{x}(s)}|\theta_z|_{\widetilde g}.
\eea
Here the norm of a differential form is
\bean
|\theta_z|_{\widetilde g}:=\max\{\theta_z(v):
\|v\|_{\widetilde g}=1\}.
\eea
\end{Def}

Given two functions $f,g:[0,\infty)\to[0,\infty)$, we write $f\lesssim g$ 
if there exists $c>0$ such that $f(s)\leq c(g(s)+1)$ for all $s\in[0,\infty)$,
and $f\sim g$ if $f\lesssim g$, $g\lesssim f$.
If we choose another Riemannian metric $g'$ on $M$ and a different base point $x'$,
then $u_{\sigma,g,x}\sim u_{\sigma,g',x'}$.
This means that the growth type of the cofilling function is an invariant of
the closed 2-form $\sigma$.

\begin{Prop}\label{prop:coh}
The growth type of the cofilling function only 
depends on the cohomology class.
\end{Prop}

\begin{proof}
Let $\sigma$, $\sigma'$ be closed cohomologous 2-forms 
on a closed manifold $M$
with $\widetilde\sigma$, $\widetilde\sigma'$ exact.
There exists a 1-form $\xi$ such that
$\sigma'=\sigma+d\xi$.
We fix $x\in\widetilde M$ and compute
\bean
u_{\sigma'}(s)
=&\inf_{\theta\in\P_{\sigma'}}\sup_{z\in B_x(s)}
|\theta_z|_{\widetilde g}\\
=&\inf_{\theta\in\P_{\sigma}}\sup_{z\in B_x(s)}
|(\theta-\widetilde\xi)_z|_{\widetilde g}\\
\leq&\inf_{\theta\in\P_{\sigma}}\sup_{z\in B_x(s)}
|\theta_z|_{\widetilde g}
+\sup_{z\in B_x(s)}
|\widetilde\xi_z|_{\widetilde g}\\
\leq&\inf_{\theta\in\P_{\sigma}}\sup_{z\in B_x(s)}
|\theta_z|_{\widetilde g}
+\max_{z\in M}
|\xi_z|_{g}\\
=& u_\sigma(s)+C,
\eea
where $C=\max_{z\in M}|\xi_z|_{g}$.
In a similar way we obtain 
$u_{\sigma}(s)\leq u_{\sigma'}(s)+C.$
Hence $u_\sigma(s)\sim u_{\sigma'}(s)$.
\end{proof}

\begin{Ex}\label{ex:tor}
Consider the standard symplectic torus 
$(\T^{2n}=\R^{2n}/\Z^{2n},\,\om=\Sigma_{i=1}^{n}dx_i\wedge dy_i)$
with the metric induced by the Euclidean metric on $\R^{2n}$.
Since $\widetilde\om=d(\Sigma_{i=1}^{n}x_i dy_i)
\in\Om^2(\R^{2n})$ and
$\Sigma_{i=1}^{n}x_i dy_i$ has linear growth with
respect to the Euclidean metric on $\R^{2n}$,
it follows that $u_{\om}(s)\lesssim s$.

For any primitive $\alpha\in\Om^1(\R^{2n})$ of $\widetilde\om$,
we obtain
\bean
\frac{\pi^n}{n!}s^{2n}
=\int_{B_x(s)}\widetilde \om=\int_{\p B_x(s)}\alpha
\leq\sup_{z\in \p B_x(s)}|\alpha_z|\cdot\int_{\p B_x(s)}1
\leq\sup_{z\in \p B_x(s)}|\alpha_z|\cdot
\frac{2\pi^n}{(n-1)!}s^{2n-1}.
\eea
This implies
$\sup_{z\in B_x(s)}|\alpha_z|\geq\frac{s}{2n}$
and hence we conclude that $u_{\om}(s)\sim s$.
\end{Ex}

\begin{Ex}\label{ex:g2}
Let $(\Sigma_g,\om)$ be a closed oriented surface 
of genus $g\geq 2$ with a volume form.
First represent $M$ as $\HH/G$,
where $\HH=\{x+yi\in\C : y>0\}$ 
is the hyperbolic upper half-plane
and $G$ is a discrete group of isometries.
Without loss of generality, we may assume that 
the lift $\widetilde \om$ of $\om$ to the universal cover $\HH$
coincides with the hyperbolic area form 
$\frac{1}{y^2}dx\wedge dy$.
Note that $\widetilde\om=d\big(\frac{1}{y}dx\big)$ and
$\frac{1}{y}dx$ is bounded with respect to the hyperbolic metric
$ds^2=\frac{1}{y^2}(dx^2+dy^2)$ on $\HH$.
By the definition, $(\Sigma_g,\om)$ is
a symplectically hyperbolic manifold and
every symplectically hyperbolic manifold
satisfies $u_{\om}(s)\sim 1$.
\end{Ex}

\begin{Prop}\label{prop:cofineq}
Let $M$, $N$ be closed manifolds 
and let $h:N\to M$ be an immersion.
Let $\sigma\in\Om^2(M)$ be a closed 2-form
which is exact on $\widetilde M$, then
$u_{h^*\sigma}(s)\lesssim u_{\sigma}(s).$
\end{Prop}

\begin{proof}
Let $g$ be a Riemannian metric on $M$,
then $h^*g$ gives a Riemannian metric on $N$.
We then have the following commutative diagram:
\bean
\xymatrix{
(\widetilde N,\widetilde h^*\widetilde g)
\ar[r]^{\widetilde h} \ar[d]^{p_N}
&(\widetilde M,\widetilde g) 
\ar[d]^{p_M}\\
(N,h^*g) \ar[r]^{h}
&(M,g)
}
\eea
where $\widetilde{[-]}$ are the lifts of $[-]$
to the universal covers.
Now choose a primitive $\theta\in\Om^1(\widetilde M)$
of $\widetilde \sigma$,
then $\widetilde h^*\theta\in\Om^1(\widetilde N)$
is a primitive of $\widetilde h^*\widetilde \sigma$.
Pick a point $z\in\widetilde N$, and note that
\bean
|(\widetilde h^*\theta)_z|_{\widetilde h^*\widetilde g}
=&\max\{(\widetilde h^*\theta)_z[v] : \|v\|
_{\widetilde h^*\widetilde g}=1,\, v\in T_z\widetilde N\}\\
=&\max\{\theta_{\widetilde h(z)}[\widetilde h_*v] : 
\|\widetilde h_*v\|_{\widetilde g}=1,\,v\in T_z\widetilde N\}\\
\leq&\max\{\theta_{\widetilde h(z)}(v) : 
\|v\|_{\widetilde g}=1,\,v\in T_{\widetilde h(z)}\widetilde M\}\\
=&|\theta_{\widetilde h(z)}|_{\widetilde g}.
\eea
We can thus estimate
\bean
u_{h^*\sigma}(s)
\sim&\inf_{\theta\in\P_{h^*\sigma}}
\sup_{z\in B_{\widetilde N,x}(s)}
|\theta_z|_{\widetilde h^*\widetilde g}\\
\lesssim&\inf_{\theta\in\P_\sigma}
\sup_{z\in B_{\widetilde N,x}(s)}
|(\widetilde h^*\theta)_z|_{\widetilde h^*\widetilde g}\\
\lesssim&\inf_{\theta\in\P_\sigma}
\sup_{z\in B_{\widetilde N,x}(s)}
|\theta_{\widetilde h(z)}|_{\widetilde g}\\
\lesssim&\inf_{\theta\in\P_\sigma}
\sup_{z\in B_{\widetilde M,\widetilde h(x)}(s)}
|\theta_z|_{\widetilde g}\\
\sim& u_\sigma(s),
\eea
which concludes the proof.

\end{proof}

\begin{Cor}\label{prop:atoroidal}
If $(M,\om)$ is a symplectically hyperbolic manifold, then
\bean
\int_{f(\T^2)}\om=0
\eea
for any immersion $f:\T^2\to M$.
\end{Cor}

\begin{proof}
If the integral is non-zero, 
then $[f^*\om]\neq 0$ in $H^2(\T^2,\R)$.
By Proposition \ref{prop:coh}, Example \ref{ex:tor} and Proposition \ref{prop:cofineq},
we deduce the following contradiction:
\bean
s\sim u_{f^*\om}(s)\lesssim u_\om(s)\sim 1.
\eea
\end{proof}

Indeed, Corollary \ref{prop:atoroidal} holds true 
for any smooth map $f:\T^2\to M$ and
the proof is almost the same, see \cite{Ked}.

\subsection{Flux homomorphism}
Let $\varphi\in{\mathit{Symp}}_0(M,\om)$
and $\{\varphi_t\}_{t\in[0,1]}$ a path of symplectomorphisms
such that $\varphi_0=\one$ and $\varphi_1=\varphi$. 
Let $Y_t$ be the generating vector field,
\bea\label{eqn:svf}
\frac{d}{dt}\varphi_t=Y_t\circ\varphi_t.
\eea
Since the Lie derivative $\L_{Y_t}\om$ vanishes, 
we get a closed 1-form $\iota_{Y_t}\om$ 
for each $t\in[0,1]$. 

\begin{Def}\label{def:flux}
The {\em flux homomorphism} 
$\widetilde{\mathit{Flux}}:\widetilde{\mathit{Symp}}_0(M,\om) 
\to H^1(M,\R)$ is defined as
\bean
\widetilde{\mathit{Flux}}\big(\{\varphi_t\}\big)=
\int_0^1\left[\iota_{Y_t}\om\right]dt,
\eea
where $[\alpha]$ is the cohomology class of a form $\alpha$.
\end{Def}
The kernel of $\widetilde{\mathit{Symp}}_0(M,\om)
\to{\mathit{Symp}}_0(M,\om)$ can be identified with
the fundamental group $\pi_1({\mathit{Symp}}_0(M,\om))$.
We denote by $\Gamma_\om\subset H^1(M,\R)$ 
the image of $\pi_1({\mathit{Symp}}_0(M,\om))$ 
by $\widetilde{\mathit{Flux}}$ 
and we call $\Gamma_\om$ 
the {\em flux group} of $(M,\om)$.
Then $\widetilde{\mathit{Flux}}$ descends to a homomorphism
\bean
\mathit{Flux}\colon \mathit{Symp}_0(M,\om)\to H^1(M,\R)/\Gamma_\om.
\eea
The next well-known fact is proved in \cite{MS}.

\begin{Thm}\label{thm:flux}
$\mathit{Ham}(M,\om)=\ker\mathit{Flux}$.
\end{Thm}

\begin{Prop}[\cite{Ked}, \cite{Pol}]\label{prop:fluxgp}
Let $(M,\om)$ be a symplectically hyperbolic manifold,
then the flux group $\Gamma_\om=0$ in $H^1(M,\R)$.
\end{Prop}

The {\em atoroidal} property of $\om$, 
$\int_{f(\T^2)}\om=0$ for any smooth $f:\T^2\to M$,
is important in the proof of Proposition \ref{prop:fluxgp}.
We then have
\bean
\mathit{Flux}\colon \mathit{Symp}_0(M,\om)\to H^1(M,\R).
\eea

\subsection{Hamiltonian structures}
Let $\Sigma$ be a closed connected orientable manifold of dimension $2n+1$.
A {\em Hamiltonian structure} on $\Sigma$ is a closed 2-form $\om$ 
such that $\om^{n}$ is nowhere vanishing.
So its kernel $\ker\om$ defines a 1-dimensional foliation
which we call the {\em characteristic foliation} of $\om$.

\begin{Def}\label{def:stb}
A Hamiltonian structure $(\Sigma,\om)$ is called {\em stable} 
if there exists a 1-form $\lambda$ 
such that
\bea\label{eqn:stb}
\ker\om\subset\ker d\lambda, \qquad \lambda\wedge\om^{n}>0.
\eea
We call the 1-form $\lambda$ a {\em stabilizing 1-form}.
This structure defines the {\em Reeb vector field R} by
\bean
\lambda(R)=1,\qquad \iota_R\om=0.
\eea
\end{Def}

\begin{Def}
A Hamiltonian structure is called {\em virtually contact}
if there is a covering $p:\widehat \Sigma\to \Sigma$ and 
a primitive $\lambda\in\Om^1(\widehat\Sigma)$ of $p^*\om$ such that
\bea\label{eqn:vcs}
\sup_{x\in\widehat{\Sigma}}|\lambda_x|\leq C<\infty,\qquad
\inf_{x\in\widehat{\Sigma}}\lambda(R)\geq\mu>0,
\eea
where $|\cdot|$ is the lifting of a metric on $\Sigma$ 
and $R$ is the pullback of a non-vanishing vector field generating $\ker\om$.
\end{Def}

\section{Hamiltonian structures on mapping tori}

\subsection{Standard and twisted Hamiltonian structures}

For a symplectically hyperbolic manifold $(M,\om)$
and a symplectomorphism 
$\varphi\in\mathit{Symp}_0(M,\om)$,
we consider the {\em mapping torus $M_\varphi$ of 
$M$ with respect to $\varphi$},
\bea\label{sec:maptorus}
M_\varphi=\frac{M\times [0,1]}{(m,0)\sim(\varphi(m),1)}.
\eea
Now we consider two Hamiltonian structures on the mapping tori $M_\one$, $M_\varphi$.
The trivial mapping torus 
$M_{\one}\cong M\times\S^1$
carries the Hamiltonian structure $\om_\one:=\pi^*\om$,
where $\pi:M_{\one}\to M:(m,\theta)\mapsto m$.
The non-trivial one $M_\varphi$ carries
the Hamiltonian structure $\om_\varphi:=\pi_\varphi^*\om$,
where $\pi_\varphi:M_\varphi\to M:(m,\theta)\mapsto m$.
The Hamiltonian structure 
$\om_\varphi$ is a well-defined 2-form on
$M_\varphi$, since $\varphi^*\om=\om$.
Note that the kernel of both Hamiltonian structures 
are spanned by ${\frac{\p}{\p\theta}}$.
Let us choose a path of symplectomorphisms 
$\{\varphi_t\}_{t\in[0,1]}$ from 
$\varphi_0=\one$ to $\varphi_1=\varphi$.
The twisting map $f:M_{\one} \to M_\varphi$ 
is defined by
\bea\label{eqn:f}
f(m,\theta)=(\varphi_\theta(m),\theta).
\eea

\begin{Lemma}\label{lem:twham}
Let $(M_\one,\om_\one)$, $(M_{\varphi},\om_\varphi)$ 
and $f:M_{\one}\to M_\varphi$ be as above,
then we obtain
\bea\label{eqn:tw-st0}
f^*\om_\varphi=\om_{\one}
-\pi^* (\iota_{Y_\theta}\om)\wedge\pi^*_{\theta}d\theta,
\eea
where $\pi_{\theta}:M_{\one}\to \S^1:(m,\theta)\mapsto \theta$.
\end{Lemma}

\begin{proof}
Let $(m_i,\theta_i)$ be tangent vectors 
in $T_{(m,\theta)}M_\one$, $i=1,2$. 
We identify $T_\theta\S^1$ with $\R$, 
so $\theta_i$ is considered as an element of $\R$.
We compute $f^*\om_\varphi$ as follows:
\bean
(f^*\om_{\varphi})_x\big(( m_1,&\theta_1),( m_2,\theta_2)\big)
=(\om_\varphi)_{f(x)}\big(f_*( m_1,\theta_1),f_*( m_2,\theta_2)\big)\\
=&(\om_\varphi)_{f(x)}\big((d\varphi_\theta(m)[ m_1]+ \theta_1\cdot Y_\theta[\varphi_\theta(m)],\theta_1),
(d\varphi_\theta(m)[ m_2]+\theta_2\cdot Y_\theta[\varphi_\theta(m)],\theta_2)\big)\\
=&(\om_\varphi)_{f(x)}\big((\theta_1\cdot Y_\theta [\varphi_\theta(m)],0),(d\varphi_\theta(m)[ m_2],\theta_2)\big)\\
&+(\om_\varphi)_{f(x)}\big((d\varphi_\theta(m)[ m_1],\theta_1),(\theta_2\cdot Y_\theta [\varphi_\theta(m)],0)\big)\\
&+(\om_\varphi)_{f(x)}\big((\theta_1\cdot Y_\theta [\varphi_\theta(m)],0),(\theta_2\cdot Y_\theta [\varphi_\theta(m)],0)\big)\\
&+\underbrace{(\om_\varphi)_{f(x)}\big((d\varphi_\theta(m)[ m_1],\theta_1),(d\varphi_\theta(m)[ m_2],\theta_2)\big).}
_{=:\diamond}\\
\eea
Here $Y_\theta$ is the vector field 
defined in (\ref{eqn:svf}).
The third summand vanishes since $\om_\varphi$ is skew-symmetric.
In order to simplify the $\diamond$-term we compute
\bea\label{eqn:dia1}
\diamond
=&(\pi_\varphi^*\om)_{f(x)}\big((d\varphi_\theta(m)[ m_1],\theta_1),(d\varphi_\theta(m)[ m_2],\theta_2)\big)\\
=&\om_{\varphi_\theta(m)}(d\varphi_\theta(m)[ m_1],d\varphi_\theta(m)[ m_2])\\
=&(\varphi_\theta^*\om)_m(m_1,m_2)\\
=&\om_m(m_1,m_2),
\eea
where the last equality comes from the assumption that
$\varphi_\theta:M\to M$ is a symplectomorphism.
We also know 
\bea\label{eqn:dia2}
(\om_\one)_x\big(( m_1,\theta_1),( m_2,\theta_2)\big)
=&(\pi^*\om)_x\big(( m_1,\theta_1),( m_2,\theta_2)\big)\\
=&\om_m(m_1,m_2).
\eea
By combining (\ref{eqn:dia1}), (\ref{eqn:dia2}) we obtain
\bea\label{eqn:twmetom}
(\om_{\one})_x\big(( m_1,\theta_1),( m_2,\theta_2)\big)
&=(\om_\varphi)_{f(x)}\big((d\varphi_\theta(m)[ m_1],\theta_1),(d\varphi_\theta(m)[ m_2],\theta_2)\big).
\eea
Hence the difference between 
$f^*\om_{\varphi}$ and $\om_{\one}$ is
\bea\label{eqn:tw-st1}
(f^*\om_{\varphi}-\om_{\one})_x\big(( m_1,\theta_1),( m_2,\theta_2)\big)
=&(\om_\varphi)_{f(x)}\big((\theta_1\cdot Y_\theta [\varphi_\theta(m)],0),(d\varphi_\theta(m)[ m_2],\theta_2)\big)\\
&+(\om_\varphi)_{f(x)}\big((d\varphi_\theta(m)[ m_1],\theta_1),(\theta_2\cdot Y_\theta [\varphi_\theta(m)],0)\big)\\
=&\om_{\varphi_\theta(m)}\big(d\varphi_\theta(m)[ m_1],\theta_2\cdot Y_\theta[\varphi_\theta(m)]\big)\\
&+\om_{\varphi_\theta(m)}\big(\theta_1\cdot Y_\theta[\varphi_\theta(m)],d\varphi_\theta(m)[ m_2]\big).
\eea
On the other hand, we have
\bea\label{eqn:tw-st2}
\big(\pi^*(\iota_{Y_\theta}\om)\wedge \pi^*_{\theta}d\theta\big)_x
&\big(( m_1,\theta_1),( m_2,\theta_2)\big)\\
=&\big(\pi^*(\iota_{Y_\theta}\om)\wedge \pi^*_{\theta}d\theta\big)_{f(x)}
\big((d\varphi_\theta(m)[ m_1],\theta_1),(d\varphi_\theta(m)[ m_2],\theta_2)\big)\\
=&(\iota_{Y_\theta}\om)_{\varphi_\theta(m)}\big(d\varphi_\theta(m)[ m_1]\big)\cdot(d\theta)_\theta(\theta_2)\\
&-(\iota_{Y_\theta}\om)_{\varphi_\theta(m)}\big(d\varphi_\theta(m)[ m_2]\big)\cdot(d\theta)_\theta(\theta_1)\\
=&\om_{\varphi_\theta(m)}\big(Y_\theta[\varphi_\theta(m)],d\varphi_\theta(m)[ m_1]\big)\cdot\theta_2\\
&-\om_{\varphi_\theta(m)}\big(Y_\theta[\varphi_\theta(m)],d\varphi_\theta(m)[ m_2]\big)\cdot\theta_1.
\eea
By combining (\ref{eqn:tw-st1}) and (\ref{eqn:tw-st2}), we conclude (\ref{eqn:tw-st0}).
This proves the lemma.
\end{proof}

\begin{Rmk}\label{rmk:ker}
The Hamiltonian structure $\om_\varphi$ on 
$M_\varphi$ has its kernel
spanned by ${\frac{\p}{\p\theta}}$. 
If there exists a closed orbit
$\gamma:\S^1\to M_\varphi$ of ${\frac{\p}{\p\theta}}$,
then its projection $\pi\circ\gamma:\S^1\to M$ gives us 
a symplectic fixed point with respect to 
$\varphi\in\mathit{Symp}_0(M,\om)$.
One can easily check that 
$(f^*\om_\varphi)({\frac{\p}{\p\theta}}-Y_\theta)=0$.
This implies that the vector field 
${\frac{\p}{\p\theta}}-Y_\theta$ spans
$\ker f^*\om_\varphi$ on $M_{\one}$.
The closed orbit of ${\frac{\p}{\p\theta}}-Y_\theta$ 
also can be interpreted as 
a fixed point of the flow of the vector field of $Y_\theta$.
\end{Rmk}

\subsection{Cofilling function of Hamiltonian structures}

\begin{Prop}\label{prop:lin}
Let $(M,\om)$ be a symplectically hyperbolic manifold and 
$\varphi\in\mathit{Symp}_0(M,\om)$.
If $\mathit{Flux}(\varphi)\neq0$
then the Hamiltonian structure $f^*\om_\varphi\in\Om^2(M_{\one})$ 
satisfies
$u_{f^*\om_\varphi}(s)\gtrsim s.$
\end{Prop}

\begin{proof}
First recall that
\bea\label{eqn:omvarphi}
f^*\om_\varphi=\om_{\one}
-\pi^* (\iota_{Y_\theta}\om)\wedge\pi^*_{\theta}d\theta.
\eea
Since our symplectic manifold $(M,\om)$ is symplectically hyperbolic,
the standard Hamiltonian structure $\om_{\one}=\pi^*\om\in\Om^2(M_{\one})$
admits a bounded primitive $\widetilde\pi^*\lambda$ on $\widetilde M\times \R$,
where $\widetilde\om=d\lambda$ and $\widetilde\pi:\widetilde M\times \R\to\widetilde M$
is the lift of $\pi$.

Now we consider the twisted term 
$\pi^* (\iota_{Y_\theta}\om)\wedge\pi^*_{\theta}d\theta$.
Since $\mathit{Flux}(\varphi)$ is nontrivial in $H^1(M,\R)$,
there exists $a\in\pi_1(M)$ such that
$\langle \mathit{Flux}(\varphi) , \overline a \rangle\neq 0$,
where $\overline a$ stands for the image of $a$ in $H_1(M,\Z)$ 
under the Hurewicz homomorphism.
Choose an immersed curve $\gamma:\S^1\to M$ 
such that $[\gamma]=a$.

Let us consider the induced immersion of $\T^2$
\bean
h:\T^2&\to M_{\one}=M\times \S^1\\
(t,\theta)&\mapsto (\gamma(t),\theta).
\eea
Then with (\ref{eqn:omvarphi}) we calculate
\bean
-\int_{h(\T^2)}f^*\om_\varphi
=&\int_{h(\T^2)}\iota_{Y_\theta}\om\wedge d\theta \\
=&\int_0^1\int_0^1\iota_{Y_\theta}\om
\left[\frac{d\gamma}{dt}\right]dt\,d\theta\\
=&\int_0^1\left(\int_0^1\iota_{Y_\theta}\om\, d\theta\right)
\left[\frac{d\gamma}{dt}\right]dt\\
=&\langle\mathit{Flux(\varphi),\overline a}\rangle\\
\neq&0.
\eea
This implies that $[h^*f^*\om_\varphi]\neq0$ 
in $H^2(\T^2,\R)$.
From Example \ref{ex:tor} and Proposition \ref{prop:cofineq},
we have 
\bean
u_{f^*\om_\varphi}(s)\gtrsim u_{h^*f^*\om_\varphi}(s)\sim s.
\eea
\end{proof}

\begin{Rmk}
Let $(M,\om)$ be a symplectically hyperbolic manifold
and $\varphi\in\mathit{Symp}_0(M,\om)$ 
with $\mathit{Flux}(\varphi)=0$.
The induced Hamiltonian structure $(M_\varphi,\om_\varphi)$
is then a new example of a virtually contact structure.

Indeed, a covering of $M_\varphi$ is given by 
\bean
\widehat M_\varphi
:=\frac{\widetilde M\times[0,1]}
{(\widetilde m,0)\sim(\widetilde \varphi(\widetilde m),1)}
\eea
with a covering map $\widehat p:\widehat M_\varphi\to M_\varphi
:(\widetilde m,\theta)\mapsto(m,\theta)$ 
where $\widetilde\varphi:\widetilde M\to\widetilde M$
is the lift of $\varphi:M\to M$. 
Note that $\widehat M_\one=\widetilde M\times \S^1$.
To obtain a primitive of the Hamiltonian structure,
we need the following notations:
\bean
\widehat f\colon&\widehat M_\one\to\widehat M_\varphi:
(\widetilde m,\theta)\mapsto(\widetilde\varphi_\theta(\widetilde m),\theta);\\
\widehat\pi\colon&\widehat M_\one\to\widetilde M:
(\widetilde m,\theta)\mapsto\widetilde m.
\eea
With this covering we obtain a primitive of
$(\widehat M_\varphi,\widehat p^*\om_\varphi)\cong
(\widehat M_\one,\widehat f^*\widehat p^*\om_\varphi)$
as follows:
\bean
\widehat f^*\widehat p^*\om_\varphi
&=\widehat p^* f^* \om_\varphi\\
&=\widehat p^*(\om_\one-\pi^*(\iota_{Y_\theta}\om)\wedge\pi^*_\theta d\theta)\\
&=d(\underbrace{\widehat\pi^*\lambda-\widehat p^* h_\theta d\theta+Kd\theta}_{=:\widehat \lambda}).
\eea
Here $\lambda$ is a bounded primitive of $\widetilde\om\in\Om^2(\widetilde M)$,
$h_\theta\in C^\infty(M)$ is a time-dependent Hamiltonian function 
corresponding to the Hamiltonian diffeomorphism $\varphi_\theta$
and $K$ is a constant fixed later on.

Since $\lambda, h_\theta$ and $K$ are bounded,
the primitive $\widehat \lambda$ is obviously bounded.
To verify the second condition 
of a virtually contact structure in (\ref{eqn:vcs}),
let $R$ be the lifted vector field of 
$\frac{\p}{\p\theta}-Y_\theta$ 
on $\widehat M_\one$.
Then we have
\bean
\widehat\lambda(R)
=K-\lambda(\widetilde Y_\theta)-\widetilde h_\theta,
\eea
where $\widetilde Y_\theta$, $\widetilde h_\theta$ are
the lifts of $Y_\theta$, $h_\theta$ to the universal cover 
$\widetilde M$.
We guarantee $\widehat \lambda(R)\geq\mu$ by taking 
\bean
K=\|\lambda_x\|_\infty\max_{\theta\in\S^1}\|Y_\theta\|_\infty
+\max_{\theta\in\S^1}\|h_\theta\|_\infty+\mu.
\eea
\end{Rmk}

\begin{Cor}\label{thm:novir}
Let $(M,\om)$ be a symplectically hyperbolic manifold 
and $(M_{\varphi},\om_{\varphi})$ be the mapping torus 
with the induced Hamiltonian structure.
If $\varphi$ is not Hamiltonian,
then $(M_\varphi,\om_{\varphi})$ admits a stable Hamiltonian structure 
but no virtually contact structure.
\end{Cor}
\begin{proof}
Let $f:M_\one\to M_\varphi$ be the twisting map defined in (\ref{eqn:f}).
As mentioned in Remark \ref{rmk:ker} $\ker(f^*\om_{\varphi})$ 
is spanned by the vector field $\frac{\p}{\p\theta}-Y_\theta$.
In order to define a stable structure on $(M_\varphi,\om_{\varphi})\cong(M\times \S^1,f^*\om_{\varphi})$,
we choose the stabilizing 1-form $\lambda$ in Definition \ref{def:stb} as $\pi_{\theta}^*d\theta$.
Since $\lambda=\pi_{\theta}^*d\theta$ is closed, the first condition in (\ref{eqn:stb}) holds trivially.
Using (\ref{eqn:omvarphi}) one verifies that 
$\lambda\wedge (f^*\om_{\varphi})^n
=\pi_{\theta}^*d\theta\wedge\pi^*\om^n$, $2n=\dim M$, 
and this form vanishes nowhere,
which implies that the second condition in (\ref{eqn:stb}) also holds true.
Thus $(M_\varphi,\om_{\varphi})$ 
admits a stable Hamiltonian structure with a stabilizing 1-form $\pi_{\theta}^*d\theta$.
By Proposition \ref{prop:lin},
the Hamiltonian structure 
$\om_\varphi$ has a cofilling function of at least linear type.
This means that there is no bounded primitive of $\om_\varphi$
even in the universal cover 
and hence $(M_\varphi,\om_\varphi)$ cannot be
a virtually contact structure.
\end{proof}

\subsection{Metrics on Hamiltonian structures}

In order to obtain primitives of the Hamiltonian structures
$(M_\one,\om_\one)$, $(M_\varphi,\om_\varphi)$,
we now consider the lifted structures 
on the universal covers.
The lifted Hamiltonian structure 
$(\widetilde M_\one,\widetilde\om_\one)$
is clearly isomorphic to 
$(\widetilde M\times\R,\widetilde\pi^*\widetilde\om)$,
where $\widetilde\pi:\widetilde M\times\R\to\widetilde M
:(\widetilde m,r)\mapsto \widetilde m$. 
Note that $\widetilde M_\varphi=\widetilde M\times\R$ 
and $\widetilde\om_\varphi=\widetilde\pi_\varphi^*\widetilde\om$,
where $\widetilde\pi_\varphi:\widetilde M_\varphi\to\widetilde M$
is the lift of $\pi_\varphi:M_\varphi\to M$.
Since $\widetilde\pi=\widetilde\pi_\varphi$,
we have $(\widetilde M_\varphi,\widetilde\om_\varphi)
\cong(\widetilde M\times\R,\widetilde\pi^*\widetilde\om)$.

Even though both lifted structures 
$(\widetilde M_\one,\widetilde\om_\one)$, 
$(\widetilde M_\varphi,\widetilde\om_\varphi)$
are isomorphic to 
$(\widetilde M\times\R,\widetilde\pi^*\widetilde\om)$,
they have different deck transformations as follows:
An element $n\in\Z\cong\pi_1(\S^1)\hookrightarrow\pi_1(M_\one)$
induces a translation 
$(\widetilde m,r)\mapsto(\widetilde m,r+n)$
on the universal cover $\widetilde M_\one$,
while $n\in\Z\cong\pi_1(\S^1)\hookrightarrow\pi_1(M_\varphi)$ act by
$(\widetilde m,r)\mapsto(\widetilde\varphi^n(\widetilde m),r+n)$
on $\widetilde M_\varphi$.
Here 
$\widetilde\varphi^n$ is the $n$-th 
iterate of $\widetilde\varphi$.

We next consider the lift 
$\widetilde f:\widetilde M_\one\to\widetilde M_\varphi$
of $f:M_\one\to M_\varphi$.
Since $\widetilde M_\one=\widetilde M\times\R
=\widetilde M_\varphi$, it suffices to define
\bean
\widetilde f:\widetilde M\times\R&\to\widetilde M\times\R\\
(\widetilde m,r)&\mapsto(\widetilde\varphi_r(m),r).
\eea
Here $\widetilde\varphi_r=\widetilde\varphi_{r-\lfloor r\rfloor}
\circ\widetilde\varphi^{\lfloor r\rfloor}$,
where $\lfloor r\rfloor$ is the largest integer 
not greater than $r$ and $\widetilde\varphi_\theta$ is the
lift of $\varphi_\theta$ for $0\leq\theta<1$.
We summarize the Hamiltonian structures, maps and their lifts
in the following diagram:
\bea\label{dia:proj}
\xymatrix{
(\widetilde M,\widetilde\om) \ar[d]^{p}
&(\widetilde M_{\one},\widetilde\om_{\one}) \ar[l]_{\widetilde\pi}
\ar[d]^{p_{\one}} \ar[r]^{\widetilde f}
&(\widetilde M_\varphi,\widetilde\om_\varphi) \ar[d]^{p_\varphi}\\
(M,\om)
&(M_{\one},\om_{\one}) \ar[l]_{\pi} \ar[r]^{f}
&(M_\varphi,\om_\varphi)
}
\eea
Here $p_{\one}$, $p_\varphi$ are the natural projections.

Now we consider Riemannian metrics on the above spaces.
Let $g$ be a Riemannian metric on $M$,
$g_{\theta}$ be the standard metric on $\S^1=\R/\Z$ 
which is induced by the Euclidean metric on $\R$.
We consider a product Riemannian metric 
$g_{\one}:=g\oplus g_{\theta}$ on $M_{\one}=M\times \S^1$ and
the lifts $\widetilde g$, $\widetilde g_{\one}$
to the corresponding universal covers 
$\widetilde M$, $\widetilde M\times\R$.

The lifted Hamiltonian structure 
$(\widetilde M_\varphi,\widetilde\om_\varphi)
\cong(\widetilde M\times\R,\widetilde\pi^*\widetilde\om)$
admits a bounded primitive 
$\widetilde\pi^*\lambda\in\Om^1(\widetilde M_\varphi)$,
i.e.
$d(\widetilde\pi^*\lambda)=\widetilde\pi^*\widetilde\om$,
with respect to $\widetilde g_\one$.
Here the primitive 1-form $\lambda\in\Om^1(\widetilde M)$
exists and is bounded with respect to $\widetilde g$,
since our manifold $(M,\om)$ is symplectically hyperbolic.
This immediately implies that 
the pull-back $\widetilde f^*\widetilde\om_\varphi$
also has a bounded primitive with respect to
the pull-back metric $\widetilde f^*\widetilde g_\one$.
By Proposition \ref{prop:lin}, however,
$\widetilde f^*\widetilde\om_\varphi$ never 
admits a bounded primitive with respect to the
metric $\widetilde g_\one$, 
when $\varphi$ is a non-Hamiltonian symplectomorphism.
Note in particular that $\widetilde f^*\widetilde g_\one$
cannot be expressed as a lift of a Riemannian metric
on $M_\one$.

We now investigate the pull-back metric
$\widetilde f^*\widetilde g_\one$ on $\widetilde M_\one$.
For $n\in\Z\subset\R$, 
$(\widetilde m,n)\in\widetilde M_\one$
and $(\widetilde m_i,0)\in T_{(\widetilde m,n)}\widetilde M_\one$,
we have
\bea\label{eqn:fmetpro}
(\widetilde f^*\widetilde g_\one)_{(\widetilde m,n)}
\big((\widetilde{ m}_1,0),(\widetilde{ m}_2,0)\big)
=&(\widetilde g_\one)_{\widetilde f(\widetilde m,n)}
\big(\widetilde f_*(\widetilde{ m}_1,0),
\widetilde f_*(\widetilde{ m}_2,0)\big)\\
=&(\widetilde g_\one)_{\widetilde f(\widetilde m,n)}
\big((d\widetilde\varphi^{n}(\widetilde m)[\widetilde{ m}_1],0),
(d\widetilde\varphi^{n}(\widetilde m)[\widetilde{ m}_2],0)\big)\\
=&\widetilde g_{\widetilde\varphi^{n}(\widetilde m)}
\big(d\widetilde\varphi^{n}(\widetilde m)[\widetilde{ m}_1],
d\widetilde\varphi^{n}(\widetilde m)[\widetilde{ m}_2]\big)\\
=&g_{\varphi^n(m)}
\big(d\varphi^n(m)[m_1],
d\varphi^n(m)[m_2]\big),
\eea
where $m=p(\widetilde m)$ and $m_i=p_*(\widetilde m_i)$.

\section{Proof of Main Theorem}

When $\varphi$ is Hamiltonian diffeomorphism on $(M,\om)$, 
as mentioned in the introduction,
Polterovich's result implies $\Gamma_n(\varphi)\gtrsim n$.
So we only need to consider
a symplectomorphism $\varphi\in\mathit{Symp}_0(M,\om)$
with a non-vanishing flux.

A crucial observation in this article is the following.
We can choose a primitive 
$\widetilde f^*\widetilde\pi^*\lambda
\in\Om^1(\widetilde M_\one)$ 
of $\widetilde f^*\widetilde\om_\varphi$ 
which is bounded with respect to the twisted metric 
$\widetilde f^*\widetilde g_\one$.
But $\widetilde f^*\widetilde\pi^*\lambda$
has at least linear growth with respect to
the standard metric $\widetilde g_{\one}$
by Proposition \ref{prop:lin}.
Now we interpret the difference between 
$\widetilde g_{\one}$ and $\widetilde f^*\widetilde g_\one$
as a lower bound for the growth rate of $\Gamma_n(\varphi)$.

Let us fix a primitive 
\bea\label{eqn:prim}
\widetilde f^*\widetilde\pi^*\lambda
=\widetilde\pi^*\lambda+r\cdot\widetilde\pi^*
(\iota_{\widetilde Y_{r}}\widetilde\om)
\eea
of $\widetilde f^*\widetilde \om_\varphi$ 
in Lemma \ref{lem:twham}.
Here $r$ is the coordinate for $\R=\widetilde\S^1$
and $\widetilde Y_r$ is the lift of $Y_\theta$ to $\widetilde M_\one$.
Without loss of generality, we may assume that
$[\iota_{Y_t}\om]\in H^1(M,\R)$ is non-trivial for $t=0$
which implies that 
$\max_{z\in\widetilde M}
|(\iota_{\widetilde Y_n}\widetilde\om)|_{\widetilde g}$
is positive for all $n\in\N\subset\R$.
Now we pick a point $m\in M$ such that 
\bean
|(\iota_{Y_0}\om)_m|_g
=|(\iota_{\widetilde Y_n}\widetilde\om)_{\widetilde m}|_{\widetilde g}
>0, \quad \forall n\in\N\subset\R.
\eea
Since $|\widetilde\pi^*\lambda|_{\widetilde g_\one}$ is bounded,
\bean
|(\widetilde f^*\widetilde\pi^*\lambda)_{(\widetilde m,n)}|
_{\widetilde{g}_\one}
\sim
n\cdot |(\iota_{\widetilde Y_n}\widetilde\om)_{\widetilde m}|
_{\widetilde g}
\sim n,
\quad \forall n\in\N\subset\R.
\eea
By definition of the norm, there is a sequence of tangent vectors 
$\{X_n\}_{n\in\N}$, 
$X_n\in T_{(\widetilde m,n)}\widetilde M_{\one}$ 
which meets the following conditions:
\begin{itemize}
\item
$\|X_n\|_{\widetilde g_{\one}}
=\|\widetilde\pi_*X_n\|_{\widetilde g}
=\|p_*\widetilde\pi_*X_n\|_{g}=1,\quad \forall n\in\N;$
\item
$(\widetilde f^*\widetilde\pi^*\lambda)
_{(\widetilde m,n)}(X_n)\sim n,$
\end{itemize}
where $\widetilde\pi:\widetilde M_\one\to\widetilde M$, 
$p:\widetilde M\to M$ are as in $(\ref{dia:proj})$.
We can assume that $X_n$ has no $r$-component, 
because $\widetilde f^*\widetilde\pi^*\lambda$ in (\ref{eqn:prim})
has no $dr$-part.

Now we change the metric $\widetilde g_\one$ to 
$\widetilde f^*\widetilde g_\one$
on $\widetilde M_{\one}$
\bean
n
\sim &(\widetilde f^*\widetilde\pi^*\lambda)_{(\widetilde m,n)}(X_n)\\
\leq &\sup_{z\in\widetilde M_\one}
|(\widetilde f^*\widetilde\pi^*\lambda)_z|
_{\widetilde f^*\widetilde g_\one}
\cdot\|X_n\|_{\widetilde f^*\widetilde g_\one} \\
= &\sup_{z\in\widetilde M_\varphi}
|(\widetilde\pi^*\lambda)_{z}|
_{\widetilde g_\one}
\cdot\|X_n\|_{\widetilde f^*\widetilde g_\one} \\
\leq & C\cdot\|X_n\|_{\widetilde f^*\widetilde g_\one},
\eea
where the constant $C>0$ comes from the fact that 
$\widetilde\pi^*\lambda$ is bounded with respect to
the metric $\widetilde g_\one$.
Since the tangent vector $X_n$ has no $r$-direction,
we can use (\ref{eqn:fmetpro}) to obtain
\bean
n^2
\lesssim&\|X_n\|^2_{\widetilde f^*\widetilde g_\one}\\
=& (\widetilde f^*\widetilde g_\one)_{(\widetilde m,n)}(X_n,X_n)\\
=& g_{\varphi^n(m)}
\big(d\varphi^n(m)[p_*\widetilde\pi_*X_n],
d\varphi^n(m)[p_*\widetilde\pi_*X_n]\big)\\
\leq& \max_{m\in M}|d\varphi^n(m)|_{g}^2
\cdot g(p_*\widetilde\pi_*X_n,p_*\widetilde\pi_* X_n)\\
=& \max_{m\in M}|d\varphi^n(m)|_{g}^2.
\eea
Hence we conclude that $\max_{m\in M}|d\varphi^n(m)|_{g}$
has at least linear growth.

\end{document}